\documentclass[reqno,11pt]{amsart}   	

\usepackage{amssymb}
\usepackage[colorlinks=false,backref=true,pagebackref=false,pdftex,pdfauthor={Wendel, Ederson},pdftitle={Monotonicity of the Morse index of radial solutions of the H\'enon equation in dimension two}]{hyperref}

\usepackage{color}
\usepackage{amsmath,amsfonts}
\usepackage{amssymb}
\usepackage{amscd}
\usepackage{enumerate}
\usepackage[utf8]{inputenc}
\usepackage[brazil, english]{babel} 
\usepackage{amsthm}
\usepackage{indentfirst}
\usepackage{latexsym,hyperref}
\usepackage[showonlyrefs]{mathtools}
\mathtoolsset{showonlyrefs=true}

\newtheorem{theorem}{{Theorem}}[section]
\newtheorem{lemma}{{Lemma}}[section]

\newtheorem{remark}{{Remark}}[section]

\newtheorem{proposition}{{Proposition}}[section]

\def\R{\mathbb R}
\def\N{\mathbb N}

\def\Z{\mathbb Z}

\newtheorem*{namedtheorem}{\theoremname}
  \newcommand{\theoremname}{testing}
  \newenvironment{named}[1]{
     \renewcommand{\theoremname}{#1}
     \begin{namedtheorem}}
     {\end{namedtheorem}}

\begin{document}

\title[Morse index of radial solutions of the H\'enon equation]{Monotonicity of the Morse index of radial solutions of the H\'enon equation in dimension two}

\thanks{Wendel Leite da Silva was partially supported by CNPq and CAPES. Ederson Moreira dos Santos was partially supported by CNPq grant 307358/2015-1 and FAPESP grant 2015/17096-6.}
\author{Wendel Leite da Silva and Ederson Moreira dos Santos}
\address{
Instituto de Ci\^encias Matem\'aticas e de Computa\c{c}\~ao \\ Universidade de S\~ao Paulo,  CEP 13560-970 - S\~ao Carlos - SP - Brazil}
\email{
wendelleite@usp.br, ederson@icmc.usp.br}

\date{\today}

\begin{abstract}
We consider the equation
\[
-\Delta u = |x|^{\alpha} |u|^{p-1}u, \ \ x \in B, \ \ u=0 \quad \text{on} \ \ \partial B,
\]
where $B \subset {\mathbb R}^2$ is the unit ball centered at the origin, $\alpha \geq0$, $p>1$, and we prove some results on the Morse index of radial solutions.  The contribution of this paper is twofold. Firstly, fixed the number of nodal sets $n\geq1$ of the solution $u_{\alpha,n}$, we prove that the Morse index $m(u_{\alpha,n})$ is monotone non-decreasing with respect to $\alpha$. Secondly, we provide a lower bound for the Morse indices $m(u_{\alpha, n})$, which shows that $m(u_{\alpha, n}) \to +\infty$ as $\alpha \to + \infty$.
\end{abstract}
\maketitle

\section{Introduction}
The Hénon equation \cite{henon} was proposed as a model to study stellar distribution in a cluster of stars with the presence of a black hole located at the center of the cluster.  Besides its application to astrophysics, Hénon-type equations also model steady-state distributions in other diffusion processes; see the introduction in \cite{pacella-indiana} and the references therein for a more precise description on applications. Apart from its applications, the Hénon equation is an excellent prototype for the study of some important problems on the qualitative analysis of solutions of elliptic partial differential equations. For example, the symmetry of least energy solutions and least energy nodal solutions \cite{smets-su-willem, BWW, SW} and some concentration phenomena \cite{cao-peng-yan, BW, pacella-indiana}. In this paper we present some results on the Morse index of radially symmetric solutions.

Consider the equation
\begin{equation}\label{1}
-\Delta u = g(|x|,u) \textrm{ in }  \Omega, \quad u =0 \textrm{ on } \partial \Omega.
\end{equation}
where $\Omega \subset \R^N$, $N\geq 2$, is either a ball or an annulus centered at the origin, $g:[0,+\infty)\times \R \rightarrow \R$ is such that $r\mapsto g(r,u)$ is $C^{0,\beta}$ on bounded sets of $[0,+\infty)\times \R$, $u\mapsto g_u(r,u)$ is $C^{0,\gamma}$ on bounded sets of $[0,+\infty)\times \R$, where $g_u$ denotes the derivative of $g$ with respect to the variable $u$. 

Given any continuous function $u : \Omega \rightarrow \R$ we will denote by $n(u)$ the number of nodal sets of $u$, i.e. of connected components of $\{x \in\Omega; u(x)\neq 0\}$. The Morse index $m(u)$ of a solution $u$ of \eqref{1} is the maximal dimension of a subspace of $H^1_0(\Omega)$ in which the quadratic form
\begin{equation*}\label{Q}
H^1_0(\Omega) \ni w\mapsto Q_u(w):=\int_{\Omega}|\nabla w(x)|^2 dx - \int_{\Omega}g_u(|x|,u(x))w^2(x) dx
\end{equation*}
is negative definite. Since we are considering the case of bounded domains, $m(u)$ coincides with the number of negative eigenvalues, counted with their multiplicity, of the linearized operator $L_u := -\Delta - g_u(|x|, u)$ in the space $H^1_0(\Omega)$. When the solution $u$ is radial, we will denote by $m_{rad}(u)$ the radial Morse index of $u$, i.e. the maximal dimension of a subspace of $H^1_{0,rad}(\Omega)$ in which the quadratic form $Q_u$ is negative definite or, alternatively, $m_{rad}(u)$ is the number of negative eigenvalues, counted with their multiplicity, of $L_u$ in the space $H^1_{0,rad}(\Omega)$.

In case the nonlinear term $g$ does not depend on the space variable, Aftalion and Pacella \cite{aftalion} obtained some lower bounds on the Morse index of sign changing radial solutions of \eqref{1}, which recently were improved by De Marchis, Ianni and Pacella \cite[Theorem 2.1]{pacella}.

\begin{named}{Theorem A}[\textbf{Autonomous problems}]\label{ThA}
Let $u$ be a radial nodal solution of \eqref{1} with $g(|x|, u) = f(u)$, $f \in C^1$. Then 
$$
m_{rad}(u) \geq n(u)-1\quad \text{and}\quad
m(u) \geq m_{rad}(u) + N(n(u)-1).
$$
Moreover, if $f$ is superlinear, i.e. satisfies the condition
\begin{equation}\label{superlinear}
f'(s)>\frac{f(s)}{s} \quad \forall\, s\in \R \backslash \{0\},
\end{equation}
then
$$
m_{rad}(u)\geq n(u) \ \ \text{and hence} \ \
m(u) \geq n(u) + N(n(u)-1).
$$
\end{named}

In this paper we consider the non-autonomous equation
\begin{equation}\label{f}
\left\{
\begin{array}{l}
\begin{aligned}
-\Delta u &= |x|^\alpha f(u) &\textrm{in}&\ \ \Omega, \vspace{0.3 cm}\\
           u &= 0                     &\textrm{on}&\ \ \partial \Omega,
\end{aligned}
\end{array}
\right.
\end{equation}
where $\Omega \subset \R^2$ is either a ball or an annulus centered at the origin, $\alpha>0$ and $f:\R \rightarrow \R$ is $C^{1,\beta}$ on bounded sets of $\R$. We recall the following estimates obtained in \cite[Theorems 1.1 and 1.2]{ederson}, which were used to prove that least energy nodal solutions of \eqref{f} are not radially symmetric.

\begin{named}{Theorem B}\label{ThB}
Let $u$ be a radial sign changing solution of \eqref{f}. Then $u$ has Morse index greater than or equal to $3$. Moreover, if \eqref{superlinear} holds, then the Morse index of $u$ is at least $n(u) + 2$. In case $\alpha$ is even, then these lower bounds can be improved, namely they become $\alpha+ 3$ and $n(u)+ \alpha+2$, respectively.
\end{named}

Very recently these lower bounds were improved in \cite[Theorem 1.1]{amadori} by characterizing the Morse index in terms of a singular one dimensional eigenvalue problem. We also mention the paper \cite{weth} where its proved that the Morse index of radial solutions goes to infinity as $\alpha \to \infty$. Given any $\beta \in \R$, we set $[\beta]:=\max \{n\in \Z; n\leq \beta\}$.

\begin{named}{Theorem C}\label{ThC}
Let $\alpha \geq0$ and $u$ be a radial nodal solution of \eqref{f}. Then 
$$
m_{rad}(u) \geq n(u)-1\quad \text{and}\quad
m(u) \geq m_{rad}(u) + \left( n (u) -1\right) \left(2\left[\frac{\alpha}{2}\right] + 2\right).
$$
Moreover, if \eqref{superlinear} holds, then
$$
m_{rad}(u)\geq n(u) \ \ \text{and hence} \ \ 
m(u) \geq n(u) + \left( n (u) -1\right) \left(2\left[\frac{\alpha}{2}\right] + 2\right).
$$
\end{named}

We obtain an improvement for these lower bounds.

\begin{theorem}\label{newtheo}
Let $\alpha \geq0$ and $u$ be a radial nodal solution of \eqref{f}. Then 
$$
m_{rad}(u) \geq n(u)-1\quad \text{and}\quad
m(u) \geq m_{rad}(u) + \left( m(u_0) - m_{rad}(u_0)\right) \left(\left[\frac{\alpha}{2}\right] + 1\right),
$$
where $u_0$ is a radial solution with $n(u_0)=n(u)$ of the autonomous problem
\begin{equation}\label{u_0}
-\Delta u_0 = \left(\frac{2}{\alpha+2}\right)^2 f(u_0) \ \  \text{in} \ \ \Omega_{\frac{2}{\alpha+2}}:=\{|x|^{\frac{\alpha}{2}}x; x\in \Omega\},\quad u_0=0 \ \ \text{on} \ \ \partial \Omega_{\frac{2}{\alpha+2}}.
\end{equation}
Moreover, if \eqref{superlinear} holds, then
$$
m_{rad}(u)\geq n(u) \ \ \text{and hence} \ \ 
m(u) \geq n(u) + \left( m(u_0) - m_{rad}(u_0)\right) \left(\left[\frac{\alpha}{2}\right] + 1\right).
$$
\end{theorem}

\begin{remark}
Observe that
\[
\frac{m(u_0)- m_{rad}(u_0)}{2} \geq n(u)-1.
\]
Indeed, the above inequality is equivalent to
\[
m(u_0) \geq m_{rad}(u_0) + 2 (n(u)-1) = m_{rad}(u_0) + 2 (n(u_0)-1), 
\]
and this is guaranteed by Theorem A.  Also observe that this inequality can be strict in case $f(s) = |s|^{p-1}s$, $p \gg1$, $n(u)=2$, since $\frac{m(u_0)- m_{rad}(u_0)}{2} =5$ by \cite[Theorem 1.1]{pacella12}.
\end{remark}

Next, we consider the particular case of the Hénon equation 
\begin{equation}\label{alpha}
\left\{
\begin{array}{l}
\begin{aligned}
-\Delta u &= |x|^{\alpha}|u|^{p-1}u &\textrm{in}&\ \ B, \vspace{0.3 cm}\\
           u &= 0                     &\textrm{on}&\ \ \partial B,
\end{aligned}
\end{array}
\right.\tag{$P_{\alpha}$}
\end{equation}
where $B \subset \R^2$ is the unit open ball centered at the origin, $\alpha \geq 0$ is a parameter and $p > 1$. Fixed the number of nodal sets $n$, we prove that the Morse index of a radial nodal solution of \eqref{alpha} with $n$ nodal sets is monotone non-decreasing with respect to $\alpha$.

\begin{theorem}[\textbf{Monotonicity of the Morse indices}]\label{theo1}
Let $u_\alpha$ and $u_\beta$ be radial solutions of \eqref{alpha} and $(P_\beta)$, respectively, with the same number $n\geq1$ of nodal sets. If $0\leq\alpha\leq\beta$, then $m(u_\alpha)\leq m(u_\beta)$ and $m_{rad}(u_\alpha)=m_{rad}(u_\beta)=n$.
\end{theorem}

The case of $N=2$ is special. We may use the change of variables \eqref{Tk} with $\kappa = \frac{\beta+2}{\alpha+2}$ to establish a correspondence between the radial solutions of \eqref{alpha} with the radial solutions of $(P_{\beta})$ with the same number of nodal sets. Although such transformation is not available for dimensions higher than two, we conjecture that Theorem \ref{theo1} should also hold for $N\geq 3$. 

\section{An auxiliary eigenvalue problem}

In this section we recall an important decomposition for some singular eigenvalue problems. Let $N\geq 2$ and consider the sphere $S^{N-1} \subset \mathbb{R}^N$. We recall that the spherical harmonics on $S^{N-1}$ are the eigenfunctions of the Laplace-Beltrami operator $-\Delta_{S^{N-1}}$. Indeed, the operator $-\Delta_{S^{N-1}}$ admits a sequence of eigenvalues $0=\lambda_0<\lambda_1\leq\ldots$ and corresponding eigenfunctions $(Y_k)$ which form a complete orthonormal system for $L^2(S^{N-1})$. More precisely, each $Y_k$ satisfies
\[
-\Delta_{S^{N-1}}Y_k(\theta)=\lambda_k Y_k(\theta),\quad \text{for}\ \theta \in S^{N-1},
\]
and each eigenvalue $\lambda_k$ is given by the formula
\begin{equation}\label{eq:SH}
\lambda_k = k(k+N-2),\quad k=0,1,\ldots
\end{equation}
whose multiplicity is
\[
N_0=1\quad \text{and}\quad N_k=\frac{(2k+N-2)(k+N-3)!}{(N-2)!k!}\ \ \text{for}\ \ k\geq1.
\]

Let $\Omega \subset \mathbb{R}^N$ be either a ball or an annulus centered at the origin and consider the problem
\begin{equation}\label{eigenfunction}
-\Delta \psi + a(x)\psi = \lambda \frac{\psi}{|x|^2}\ \textrm{ in }\ \Omega\backslash\{0\}, \quad \psi = 0 \textrm{ on }  \partial \Omega,
\end{equation}
where $a(x)$ is a radial function in $L^{\infty}(\Omega)$. Set
\[
\mathcal{H}_0 := \left\{ u \in H^1_0(\Omega);  \int_{\Omega} \frac{u^2}{|x|^2} < \infty\right\}.
\]
Then, endowed with inner product
\[
(u,v)_{\mathcal{H}_0 }:= \int_\Omega \nabla u \nabla v + \frac{uv}{|x|^2} dx, \quad u, v \in \mathcal{H}_0,
\]
$\mathcal{H}_0$ is a Hilbert space. We say that $\psi \in \mathcal{H}_0\backslash\{0\}$ is an eigenfunction of \eqref{eigenfunction}, if
\[
\int_{\Omega} \nabla \psi \nabla \varphi + a(x) \psi \varphi dx = \lambda  \int_{\Omega} \frac{\psi \varphi}{|x|^2} dx \ \ \forall\, \varphi \in \mathcal{H}_0.
\]

We recall the following result on the decomposition of eigenvalues of \eqref{eigenfunction}; see \cite[Proposition 4.1]{amadori} or \cite[Lemma 3.1]{bartsch}.

\begin{proposition}\label{decomposition}
Let $\lambda < \left(\frac{N-2}{2}\right)^2$ be an eigenvalue of \eqref{eigenfunction}.
Then, there exists $k\geq0$ such that 
\begin{equation}\label{lambda}
\lambda=\lambda_{rad}+\lambda_k\, ,
\end{equation}
where $\lambda_{rad}$ is a radial eigenvalue of \eqref{eigenfunction} and $\lambda_k$ as in \eqref{eq:SH}. Conversely, if \eqref{lambda} holds and $\psi_{rad}$ is an eigenfunction associated to $\lambda_{rad}$, then $\psi=\psi_{rad}(r)Y_k(\theta)$ is an eigenfunction of \eqref{eigenfunction}
associated to $\lambda$. 
\end{proposition}

\section{Proofs of the main results}

Given $\kappa>0$, set
\begin{equation}\label{Tk}
T_\kappa:\R^2\rightarrow \R^2, \quad T_\kappa(0)= 0 \ \ \text{and} \ \ 
T_\kappa(y):=|y|^{\kappa-1}y \ \ \text{for} \ \ y \neq0.
\end{equation}
\noindent We perform the change of variable $x\leftrightarrow y$ putting $x=T_\kappa(y)$ and we observe that, see \cite[Lemma 2.1]{ederson}, $T_{\kappa}$ has the following properties:
\begin{itemize}
\item[(i)] $T_\kappa$ is a diffeomorphism between $\R^2\backslash \{0\}$ and $\R^2\backslash \{0\}$ whose inverse is
\begin{equation}\label{Tk-1}
T_\kappa^{-1}(x)=|x|^{\frac{1}{\kappa}-1}x,\ \ i.e.\ \ T_\kappa^{-1}=T_{\frac{1}{\kappa}}.
\end{equation}
\item[(ii)] In cartesian coordinates,
\begin{equation}\label{JTk}
|\det J_{T_\kappa}(y)|=\kappa|y|^{2\kappa-2},\ \ \forall \, y\neq 0.
\end{equation}

\end{itemize}

Let $\Omega \subset \R^2$ be either a annulus or a ball centered at the origin and set $\Omega_\kappa:= T_\kappa^{-1}(\Omega)$.
\begin{lemma}[Lemma 2.4 in \cite{ederson}]\label{lemma 1}
The map
$$
S_\kappa:H^{1}_{0}(\Omega_\kappa)\rightarrow H^1_0(\Omega),\ \ S_\kappa \psi:=\psi\circ T_\kappa^{-1},
$$
is a continuous linear isomorphism. Moreover, with $\varphi=\psi\circ T_\kappa^{-1}$,
$$
\min\left\{\kappa,\frac{1}{\kappa}\right\}\int_{\Omega}|\nabla\varphi(x)|^2dx \leq \int_{\Omega_\kappa}|\nabla\psi(y)|^2dy \leq \max\left\{\kappa,\frac{1}{\kappa}\right\}\int_{\Omega}|\nabla\varphi(x)|^2dx, \ \ \forall \, \psi \in H^1_0(\Omega_\kappa),
$$
\[
\kappa \int_{\Omega} |\nabla \varphi (x)|^2 dx = \int_{\Omega_{\kappa}} |\nabla \psi(y)|^2 dy,  \ \ \forall \ \psi \in H^1_{0,{\rm{rad}}}(\Omega_{\kappa}).
\]

\end{lemma}

Given a radial function $u : \Omega \to \R$, set  $v : \Omega_\kappa \to \R$ by $v(y) = u(T_\kappa y)$. Then $v$ is radially symmetric and, by \cite[eq. (2.13)]{ederson},
\begin{equation}\label{radiallaplacian}
\Delta v(y) = \kappa^2 |x|^{\frac{2\kappa - 2}{\kappa}}\Delta u(x).
\end{equation}
Thus, if $u$ is a radial solution of \eqref{f}, then $v : \Omega_\kappa \to \R$  satisfies
$$
-\Delta v(y) = \kappa^2 |y|^{2\kappa-2+\kappa \alpha} f(v(y)), \quad y \in \Omega_\kappa, \quad v = 0 \quad \rm{on}\ \partial \Omega_\kappa.
$$
Now, given any $\beta \geq 0$, we choose $\kappa$ so that
\begin{equation}\label{kab}
2\kappa-2+\kappa \alpha = \beta, \quad \rm{i.e.}\quad \kappa := \frac{\beta+2}{\alpha+2}\,.
\end{equation}
and so
\begin{equation}\label{beta}
-\Delta v(y) = \left(\frac{\beta+2}{\alpha+2}\right)^2 |y|^{\beta} f(v(y)), \quad y \in \Omega_\kappa, \quad v = 0 \quad \rm{on}\ \partial \Omega_\kappa.
\end{equation}
In the particular case with $f(s)=|s|^{p-1}s$, setting $w(y) = \left(\frac{\beta+2}{\alpha+2}\right)^{\frac{2}{p-1}}v(y)$, we get
$$
-\Delta w(y) = |y|^{\beta} |w(y)|^{p-1}w(y), \quad y \in \Omega_\kappa, \quad v = 0 \quad \rm{on}\ \partial \Omega_\kappa.
$$
Therefore, we have proved the following result.

\begin{lemma}\label{characterization} 
$u_\alpha$ is a radial solution of \eqref{alpha} in $\Omega$ with $n$ nodal sets if, and only if, 
\begin{equation}\label{eq:related}
u_{\beta}(y) = \left(\frac{\beta+2}{\alpha+2}\right)^{\frac{2}{p-1}}u_\alpha\left(|y|^{\frac{\beta-\alpha}{\alpha+2}}y\right),\quad y \in \Omega_\kappa,\quad \kappa = \frac{\beta+2}{\alpha+2},
\end{equation}
is a radial solution of $(P_\beta)$ in $\Omega_\kappa$ with $n$ nodal sets.
\end{lemma}

Given $\alpha \geq0$, $n \in \N$, we know that there exists a unique solution of \eqref{alpha} (up to multiplication by $-1$) with $n$ nodal sets; see \cite[Theorem 1.3 (i)]{ederson}. 
Let $u_\alpha$ and $u_\beta$ be radial solutions of \eqref{alpha} and $(P_\beta)$, respectively, with $n(u_\alpha)=n(u_\beta)$. Then $u_\alpha$ and $u_\beta$ are related by \eqref{eq:related} and $m(u_\alpha)$ is the maximal dimension of a subspace of $H^1_0(B)$ in which the quadratic form
\begin{equation*}
H^1_0(B) \ni w\mapsto Q_\alpha(w):=\int_{B}|\nabla w(x)|^2 dx - p\int_{B}|x|^\alpha |u_\alpha(x)|^{p-1}w^2(x) dx
\end{equation*}
is negative definite. Similarly we can compute $m(u_\beta)$. The crucial point for the proof of Theorem \ref{theo1} is the following result.
\begin{proposition}\label{prop1}
If $0 \leq \alpha \leq \beta$, then
$$
Q_{\beta}(w_\kappa)\leq \kappa \, Q_{\alpha}(w),\quad \forall \,  w \, \in H^1_0(B), \ \ Q_{\beta}(w_\kappa)= \kappa \, Q_{\alpha}(w),\quad \forall \,  w \, \in H^1_{0, rad}(B),
$$
where $w_\kappa(y) = w \circ T_\kappa(y)$, $T_\kappa$ is defined as in \eqref{Tk}  with $\kappa = \frac{\beta+2}{\alpha+2}$. 
\end{proposition}

\begin{proof}[\bf{Proof}]
By \eqref{eq:related}, up to multiplication by $-1$,
$$
u_{\beta}(y) = \kappa^{\frac{2}{p-1}}u_\alpha (T_\kappa y).
$$
Hence
$$
\begin{aligned}
Q_{\beta}(w_\kappa) &= \int_{B}|\nabla w_\kappa(y)|^2dy - \kappa ^2 p\int_{B}|y|^{\beta} |u_\alpha(T_\kappa y)|^{p-1}w_\kappa^2(y)dy.
\end{aligned}
$$
Since $\kappa \geq 1$, it follows from Lemma \ref{lemma 1} that
$$
\int_{B}|\nabla w_\kappa(y)|^2dy \leq \max\left\{\kappa,\frac{1}{\kappa}\right\}\int_{B}|\nabla w(x)|^2dx = \kappa \int_{B}|\nabla w(x)|^2dx \ \ \forall \,  w \, \in H^1_0(B),
$$
$$
\int_{B}|\nabla w_\kappa(y)|^2dy = \kappa \int_{B}|\nabla w(x)|^2dx \quad \forall \,  w \, \in H^1_{0, rad}(B).
$$
Now, putting $x=T_\kappa(y)$, it follows from \eqref{Tk-1} and \eqref{JTk} that $dy = \kappa^{-1}|x|^{\frac{2-2\kappa}{\kappa}}dx$. Thus
$$
\kappa ^2 p \! \int_{B}|y|^{\beta} |u_\alpha(T_\kappa y)|^{p-1}w_\kappa^2(y)dy \!= \!\kappa p \!\int_{B}|x|^{\frac{\beta-2\kappa+2}{\kappa}} |u_\alpha(x)|^{p-1}w^2(x)dx \!=\!\kappa p \! \int_{B}|x|^\alpha |u_\alpha(x)|^{p-1}w^2(x)dx,
$$
since $\frac{\beta-2\kappa+2}{\kappa}=\alpha $. Therefore,
\[
Q_{\beta}(w_\kappa) \leq \kappa \int_{B}|\nabla w(x)|^2dx - \kappa p \int_{B}|x|^\alpha |u_\alpha(x)|^{p-1}w^2(x)dx =\kappa \, Q_{\alpha}(w) \ \ \forall \,  w \, \in H^1_0(B)   
\]
and
\[
Q_{\beta}(w_\kappa) = \kappa \, Q_{\alpha}(w) \ \ \forall \,  w \, \in H^1_{0,rad}(B). \qedhere
\]
\end{proof}

\begin{proof}[\bf{Proof of Theorem \ref{theo1}}]
Let $u_\alpha$ and $u_{\beta}$ be radial solutions of \eqref{alpha} and $(P_{\beta})$, respectively, with $n(u_\alpha)=n(u_\beta)=n$, $\alpha \leq \beta$. From Proposition \ref{prop1}, we have
$$
Q_{\beta}(w_\kappa)\leq \kappa \, Q_{\alpha}(w),\quad \forall \,  w \, \in H^1_0(B), \ \ Q_{\beta}(w_\kappa)= \kappa \, Q_{\alpha}(w),\quad \forall \,  w \, \in H^1_{0, rad}(B).
$$
Therefore, if $V$ is a subspace of $H^1_0(B)$ in which the quadratic form $Q_{\alpha}$ is negative definite, then $Q_{\beta}$ is also negative definite in the subspace $V_{\kappa}:=\{w\circ T_\kappa; w\in V\}$. Moreover, $Q_{\alpha}$ is negative definite in a subspace $V$ of $H^1_{0,rad}(B)$ if, and only if, $Q_{\beta}$ is negative definite in subspace $V_{\kappa}$. Since $V$ and $V_{\kappa}$ have the same dimension, we infer that $m(u_\alpha)\leq m(u_\beta)$ and $m_{rad}(u_\alpha)=m_{rad}(u_\beta)$. Finally, we know from \cite[Proposition 2.9]{Harrabi} that $m_{rad}(u_0)=n$ and we conclude the proof of Theorem \ref{theo1}.
\end{proof} 

To prove Theorem \ref{newtheo}, we start with the particular case with $\alpha$ even.

\begin{proposition}\label{prop2}
Let $\alpha > 0$ be even and let $u$ be a radial nodal solution of \eqref{f}. Then
$$
m_{rad}(u)\geq n(u)-1\quad \text{and}\quad
m(u) \geq m_{rad}(u) + (m(u_0)-m_{rad}(u_0))\left(\frac{\alpha + 2}{2}\right),
$$
where $u_0$ is a radial solution of \eqref{u_0} with $n(u_0)=n(u)$. If in addition \eqref{superlinear} holds, then
\begin{equation*}\label{estimate}
m_{rad}(u)\geq n(u)\quad \text{and hence}\quad
m(u) \geq n(u)+(m(u_0)-m_{rad}(u_0))\left(\frac{\alpha + 2}{2}\right).
\end{equation*}
\end{proposition}

\begin{proof}[{\bf Proof.}]
Let $u$ be a radial nodal solution of \eqref{f},
with $\alpha = 2(m-1)$ even and $\kappa=\frac{1}{m}$. Then $\kappa = \frac{2}{\alpha+2}$ and, by \eqref{beta} ($\beta=0$ in this case), the function $u_0=u\circ T_{\kappa}$ is a radial nodal solution of the autonomous problem
\[
-\Delta u_0 = \frac{1}{m^2} f(u_0) \ \textrm{ in } \ \Omega_\kappa, \quad u_0 = 0  \ \textrm{ on }\ \partial \Omega_\kappa.
\]
Therefore, by \cite[Lemma 2.6]{gladiali}, the singular eigenvalue problem
\begin{equation}\label{eig-no-weight}
-\Delta \psi - \frac{1}{m^2} f'(u_0)\psi = \lambda \frac{\psi}{|y|^2} \ \textrm{ in } \ \Omega_\kappa \backslash\{0\}, \quad \psi = 0 \ \textrm{ on } \ \partial \Omega_\kappa,
\end{equation}
has $m(u_0)-m_{rad}(u_0)$ negative eigenvalues associated to nonradial eigenfunctions, counted with their multiplicity. We count the negative radial eigenvalues of \eqref{eig-no-weight} as $\lambda_{1}=\lambda_{1}^{rad} < \lambda_{2}^{rad}<\ldots<\lambda_{m_{rad}(u_0)}^{rad}$ whose corresponding eigenfunctions we denote by $\psi_n^{rad}$. By Proposition \ref{decomposition}, every negative nonradial eigenvalue of \eqref{eig-no-weight} has the decomposition $\lambda_n^{rad}+k^2$, for some $n=1,\ldots,m_{rad}(u_0)$ and $k\in\N$. For each $n=1,\ldots,m_{rad}(u_0)$, consider $\mathcal{N}_n:=\{k\in\N; \lambda_n^{rad}+k^2<0\}$. Thus
\begin{equation}\label{eq:DES1}
m(u_0)-m_{rad}(u_0) = 2\sum_{n=1}^{m_{rad}(u_0)}\mathcal \# \mathcal{N}_n.
\end{equation}
Moreover, using eq. \eqref{radiallaplacian} with $\kappa=\frac{2}{\alpha+2}$, we have that the functions $\varphi_n^{rad}=\psi_n^{rad}\circ T_{\kappa}^{-1}$, $n=1,\ldots,m_{rad}(u_0)$, are the radial eigenfunctions of

\begin{equation}\label{eigenvalue}
-\Delta \varphi - |x|^\alpha f'(u)\varphi = \lambda \frac{\varphi}{|x|^2} \ \textrm{ in } \ \Omega\backslash\{0\}, \quad \varphi = 0 \ \textrm{ on } \ \partial \Omega,
\end{equation}
with $\lambda = m^2\lambda_n^{rad}<0$. Then $m_{rad}(u)=m_{rad}(u_0)$ and with $\mathcal{N}_n^{\alpha}:=\{k\in\N; m^2\lambda_n^{rad}+k^2<0\}$, 
\begin{equation} \label{eq:DES2}
m(u)-m_{rad}(u) = 2\sum_{n=1}^{m_{rad}(u_0)}\mathcal \# \mathcal{N}_n^{\alpha}\, .
\end{equation}

We claim that
\begin{equation}\label{eq:DES3}
\# \mathcal{N}_n^{\alpha} \geq m (\# \mathcal{N}_n) \quad \forall\, n= 1, \ldots, m_{rad}(u_0).
\end{equation}
Indeed, if $k \in \mathcal{N}_n$, then $\lambda_n^{rad} + k^2 < 0$, whence $m^ 2\lambda_n^{rad} + (mk)^2 < 0$. The latter shows that $mk \in \mathcal{N}_n^{\alpha}$ and this proves \eqref{eq:DES3}.

Therefore, from \eqref{eq:DES1}, \eqref{eq:DES2}, \eqref{eq:DES3}, we infer that
\[
m(u) - m_{rad}(u) \geq (m(u_0)-m_{rad}(u_0)) m = (m(u_0)-m_{rad}(u_0))\left(\frac{\alpha + 2}{2}\right).
\]

Now, with respect to the radially symmetric eigenfunctions, it is proved \cite[Theorem 2.1]{pacella} that \eqref{eig-no-weight} has at least $n(u)-1$ negative eigenvalues associated to radial eigenfunctions, and this number becomes $n(u)$ if \eqref{superlinear} holds. Again using eq. \eqref{radiallaplacian} with $\kappa=\frac{2}{\alpha+2}$, we have that $\lambda\mapsto m^2 \lambda$ is a bijection between radial eigenvalues of \eqref{eig-no-weight} and we obtain the lower bounds for $m_{rad}(u)$.
\end{proof}

Next we use Proposition \ref{prop2} to prove Theorem \ref{newtheo}. 

\begin{proof}[\bf{Proof of Theorem \ref{newtheo}}]
Let $\alpha>0$ and $u$ be a radial nodal solution of \eqref{f}. Then, by \eqref{beta}, for all $\gamma\geq 0$, the function $v=u\circ T_\kappa$ is a radial nodal solution of
\[
-\Delta v = \kappa^2 |y|^{\gamma} f(v) \ \textrm{ in } \ \Omega_\kappa, \quad v = 0 \   \textrm{ on } \ \partial \Omega_\kappa, \ \ \textrm{ with } \ \kappa = \frac{\gamma+2}{\alpha+2}.
\]
Hence, if $\gamma \leq \alpha$, i.e. $\kappa \leq 1$, setting $w_\kappa = w\circ T_\kappa$, it follows from Lemma \ref{lemma 1} and \eqref{JTk} that
$$
\int_{\Omega}|\nabla w(x)|^2 dx - \int_{\Omega}|x|^{\alpha}f'(u(x))w^2(x) dx \leq \frac{1}{\kappa} \left[\int_{\Omega_\kappa}|\nabla w_\kappa(y)|^2 dy - \kappa^2 \int_{\Omega_\kappa}|y|^{\gamma}f'(v(y))w^2(y) dy \right],
$$
for all $w \in H^1_0(\Omega)$ and the equality holds for all $w \in H^1_{0, rad}(\Omega)$. Consequently, $m_{rad}(u) = m_{rad}(v)$ and $m(u) \geq m(v)$. In particular, taking $\gamma = 2[\frac{\alpha}{2}]$ we can use Proposition \ref{prop2} for $v$ to obtain
\[
m_{rad}(v) \geq n(v)-1=n(u)-1,\ \ m_{rad}(v) \geq n(u)\ \text{if \eqref{superlinear} holds} \ \ \text{and}
\]
\[
m(v) \geq m_{rad}(v) + (m(u_0)-m_{rad}(u_0))\left(\frac{\gamma+2}{2}\right) = m_{rad}(u) + (m(u_0)-m_{rad}(u_0))\left(\left[\frac{\alpha}{2}\right] + 1\right),
\]
where $u_0$ is radial solution of \eqref{u_0} with $n(u_0)=n(u)$.
\end{proof}

\begin{remark}
Observe that the key argument in the proof of Theorem \ref{newtheo} is the monotonicity of the Morse indices $m(u) \geq m(v)$ proved above, thanks to $\gamma \leq \alpha$.
\end{remark}


\end{document}